\documentclass[a4paper,11pt,reqno]{amsart}
\usepackage{amsmath,amssymb,url,comment}

\newcommand{\Z}{\mathbb{Z}}
\newcommand{\N}{\mathbb{N}}
\newcommand{\Q}{\mathbb{Q}}

\newcommand{\calA}{\mathcal{A}}
\newcommand{\calE}{\mathcal{E}}

\newcommand{\calU}{\mathcal{U}}

\newcommand{\calP}{\mathcal{P}}

\newtheorem{Lem}{Lemma}

\newtheorem{Prop}{Proposition}

\newcommand{\qf}[1]{\langle #1 \rangle}
\newcommand{\gen}{\operatorname{gen}}

\newcommand{\ord}{\operatorname{ord}}

\title[Integral quadratic forms avoiding arithmetic progressions]%
{Integral quadratic forms avoiding arithmetic progressions}%

\author[A.G. Earnest]{A.G. Earnest}
\address{Department of Mathematics, Southern Illinois University, Carbondale, IL, 62901, U.S.A.}
\email{aearnest@siu.edu}

\author[Ji Young Kim]{Ji Young Kim}
\address{Department of Mathematical Sciences, Seoul National University,
1 Gwanak-ro, Gwanak-Gu, Seoul, 08826, Korea}%
\email{jykim98@snu.ac.kr}

\thanks{The second named author was supported by Basic Science Research Program through the National Research Foundation
of Korea (NRF) funded by the Ministry of Education (NRF-2017R1D1A1B03028905).}


\date{September 18, 2019}

\begin{document}

\begin{abstract}

For every positive integer $k$, it is shown that there exists a positive definite diagonal quaternary integral quadratic form that represents all positive integers except for precisely those which lie in $k$ arithmetic progressions. For $k=1$, all forms with this property are determined.

\end{abstract}

\maketitle

\section{Introduction}

In a recent paper \cite{kW-2018}, Williams posed and answered a variety of interesting questions regarding the integers represented by diagonal positive definite quaternary integral quadratic forms.
Among them are several which relate to the avoidance of arithmetic progressions by the represented value set of a quadratic form of this type.
For example, can such a form represent all positive integers except for those in an arithmetic progression? What about two arithmetic progressions?
Both of these questions are answered in the affirmative in \cite{kW-2018} by producing specific forms having these properties.
The proofs given there are elementary and are based on the classical theorem of Lagrange on sums of four squares.

It is natural to extend this line of inquiry and ask whether it is possible for a diagonal positive definite quaternary integral quadratic form to represent all positive integers except for those in exactly $k$ arithmetic progressions, for any positive integer $k$.
In this paper, we will show that this is indeed possible for all positive integers $k$.
In the process, we will investigate more generally, for positive definite integral quadratic forms of rank at least $4$, when the represented value set of such a form can avoid precisely a union of arithmetic progressions.
We explain how this phenomenon of avoiding arithmetic progressions arises from local obstructions to representation, and point out the connection with the theory of forms that are regular, in the sense of Dickson \cite{leD-26}.

\section{Preliminaries}

Throughout the remainder of the paper we will adopt the modern geometric language of quadratic spaces and lattices. Our terminology and notation will follow that of \cite{otO-1973}, unless otherwise indicated. The symbols $\Q$, $\Z$ and $\N$ will denote the field of rational numbers, the ring of rational integers, and the set of natural numbers, respectively. Throughout the paper, the term \textit{lattice} will always refer to a finitely generated submodule $L$ on a finite dimensional quadratic space $(V,Q)$ over $\Q$. Moreover, it will be assumed throughout that the scale of any lattice $L$ under consideration (that is, the fractional ideal generated by the set $\{B(x,y):x,y \in L\}$, where $B$ is the symmetric bilinear form on $V$ for which $Q(v)=B(v,v)$ for $v\in V$) is equal to $\Z$. So the determinant of the Gram matrix of $L$ with respect to any basis is an integer, which will be denoted by $dL$. The notation $L\cong \qf{a_1,\ldots,a_n}$ means that $L$ has an orthogonal basis with respect to which the Gram matrix of $L$ is the diagonal matrix with diagonal entries $a_1,\ldots,a_n$.

The represented value set for the lattice $L$ is $Q(L)=\{Q(v):v \in L\}$. So an integer $a$ is represented by $L$ if and only if $a\in Q(L)$. The lattice $L$ is positive definite if $Q(v)>0$ for all $0\neq v \in L$. For a positive definite lattice $L$ we will use the notation $\calE(L)$ to denote the excluded set $\N\setminus Q(L)$.

For a prime $p$, $\Z_p$ will denote the ring of $p$-adic integers and $\Z_p^{\times}$ the group of units of $\Z_p$. So for each $\alpha \in \Z_p$, there exists a unit $\alpha_0 \in \Z_p^{\times}$ such that $\alpha = p^{\ord_p\alpha}\alpha_0$. For a lattice $L$, let $\gen L$ be the genus of $L$, and let $Q(\gen L) = \{a\in \Z : a \in Q(K) \text{ for some } K\in \gen L\}$. Then $a \in Q(\gen L)$ if and only if $a \in Q(L_p)$ for all primes $p$. For a positive definite lattice $L$, let $\calE(L) = \N\setminus Q(L)$, $\calE(\gen L)= \N\setminus Q(\gen L)$ and $\calE(L_p)=\Z_p\setminus Q(L_p)$. Then $Q(L) \subseteq Q(\gen L)$ and $\calE(\gen L) \subseteq \calE(L)$. In terminology first introduced by Dickson \cite{leD-26}, a positive definite lattice $L$ is called \textit{regular} if $Q(L) = Q(\gen L)$ (or, equivalently, $\calE(L) = \calE(\gen L)$) .

\section{Arithmetic progressions and local obstructions}

For positive integers $a<m$, let $\calA_{a,m}$ denote the arithmetic progression $\{ a + m x : x \in \N \cup \{0\} \}$. For the purpose of this paper, we will refer to the arithmetic progression $\calA_{a,m}$ as \textit{admissible} if
    \[
    \ord_p a < \ord_p m
    \]
for all primes $p$ such that $p \mid m$. Our goal in this section and the next will be to characterize those positive definite lattices for which $\calE(L)$ is a union of finitely many admissible arithmetic progressions. In the final section of the paper, we will return to the question of how many arithmetic progressions are required in such an expression. For this question, some care must be taken in counting since, for example,
    \[
    \calA_{3,4} = \calA_{3,8}\cup \calA_{7,8} = \calA_{3,16}\cup \calA_{7,16} \cup \calA_{11,16}\cup \calA_{15,16},
    \]
and so on. To make the counting precise, for any $k \in \N$ we will say that a positive definite lattice $L$ \textit{avoids exactly $k$ arithmetic progressions} if $\calE (L)$ is a union of $k$ admissible arithmetic progressions, but no such expression exists for a union of fewer than $k$ arithmetic progressions.

Now let $L$ be a lattice with rank $L \geq 4$ and let $p$ be a prime. Note first that if $p\nmid dL$, then $L_p$ is unimodular and $Q(L_p)=\Z_p$. If $p\mid dL$, let $r:=r(p)$ be the positive integer such that $p^r\Z_p$ is the scale of the final component in a Jordan splitting of $L_p$. Then, using the condition that rank $L \geq 4$, it follows from the criteria in \cite[Theorem 1]{otO-1958} if $p$ is odd, or \cite[Theorem 3]{otO-1958} if $p=2$, that $p^r\Z_p \subseteq Q(L_p)$.

As the elements of $\Z_p$ of a given order lie in exactly two squareclasses modulo $(\Z_p^{\times})^2$ if $p$ is odd, and exactly four squareclasses modulo $(\Z_2^{\times})^2$ if $p=2$, it then follows that for each prime $p$ there are at most finitely many squareclasses modulo $(\Z_p^{\times})^2$ that fail to be represented by $L_p$. Now consider one such squareclass $a(\Z_p^{\times})^2$ with $a\in \Z$ and write $a=p^{\ell}b$, where $p\nmid b$. Let $\overline{b}$ denote the residue of $b$ modulo $p$ (or modulo $8$ when $p=2$). Consider first the case of an odd prime $p$. Then, by the Local Square Theorem \cite[63:1]{otO-1973}, $b$ is a square (nonsquare, respectively) modulo $p$ if and only if $\overline{b}$ lies in a system $\{ \varepsilon_1,\ldots,\varepsilon_{\frac{p-1}{2}}\}$ of quadratic residues (nonresidues, respectively) modulo $p$. That is,
    \[
    \calA_{p^{\ell}\varepsilon_1,p^{\ell+1}} \cup\cdots\cup \calA_{p^{\ell}\varepsilon_{\frac{p-1}{2}},p^{\ell+1}} = a(\Z_p^{\times})^2\cap \Z.
    \]
So, whenever $a\in \calE(L_p)$, all elements of this union are excluded from representation by $L_p$, hence from representation by $\gen L$. That is,
    \[
    \calA_{p^{\ell}\varepsilon_1,p^{\ell+1}} \cup\cdots\cup \calA_{p^{\ell}\varepsilon_{\frac{p-1}{2}},p^{\ell+1}} \subseteq \calE(\gen L).
    \]
When $p=2$, we have $\overline{b} \in \{1,3,5,7\}$ and
    \[
    \calA_{2^{\ell}\overline{b},2^{(\ell + 3)}} = a(\Z_2^{\times})^2\cap \Z.
    \]
Repeating this process for each of the finitely many primes $p$ dividing $dL$ and each of the excluded squareclasses modulo $p$, we ultimately obtain a union $\calU$ of admissible arithmetic progressions which is contained in $\calE(\gen L)$ and which accounts for all of the local obstructions to representation. Since any positive integer not in $Q(\gen L)$ must be excluded by failure to be represented by $L_p$ for at least one such prime $p$, in fact we then have $\calE(\gen L)=\calU$.

The preceding discussion is summarized in the following statement:

\begin{Prop}
If $L$ is a positive definite lattice with rank $L \geq 4$, then $\calE(\gen L)$ is a finite union of admissible arithmetic progressions.
\end{Prop}

\section{Connection to regularity}

It follows immediately from Proposition 1 that if the positive definite lattice $L$ is regular, then $\calE(L)$ is a finite union of admissible arithmetic progressions. The primary goal of this section is to prove that the converse is also true. In order to do this, we will first prove two lemmas regarding the arithmetic progressions contained in $Q(\gen L)$ and $\calE(L)$

\begin{Lem}
Let $L$ be a positive definite lattice with rank $L \geq 4$ and let $\calA_{a,m}$ be an admissible arithmetic progression.
If $\calA_{a,m} \subseteq Q(\gen L)$, then $\calA_{a,m} \cap Q(L) \not= \emptyset$.
\end{Lem}

\begin{proof}
Let $\gcd(a,m) = d$ and write $a=da'$, $m=dm'$, $\gcd(a',m') =1$.
Then $\calA_{a,m} = d ~\calA_{a',m'}$. Since $\gcd(a',m') =1$, by Dirichlet's theorem on primes in an arithmetic progression, the set
$\calP' = \{ \lambda \in \calA_{a',m'} : \lambda \text{ is prime}\}$ is infinite. Therefore the set $\calP = \{ d\lambda : \lambda \in \calP \}$
is an infinite subset of $\calA_{a,m}$ all of whose elements are divisible by bounded powers of all primes.
Since $\calP \subseteq Q(\gen L)$, it follows from \cite[Theorem 1]{aeR-gP-1946} that all sufficiently large elements of $\calP$ are represented by $L$, which yields the desired result.
\end{proof}

\noindent \textit{Remark:} If rank $L \geq 5$ (or if rank $L = 4$ and $L_p$ is isotropic for all primes $p$), then the conclusion above follows directly from a theorem of Tartakowsky \cite{vT-1929} (see, e.g., \cite[Theorem 1.6, p.204]{jwsC-1978} or \cite[Theorem 6.6.2]{yK-1993}), without invoking Dirichlet's theorem.

\begin{Lem}
Let $L$ be a positive definite lattice with rank $L \geq 4$ and let $\calA_{a,m}$ be an admissible arithmetic progression.
If $\calA_{a,m} \subseteq \calE(L)$, then $a \in \calE(L_p)$ for some prime $p \mid m$.
\end{Lem}

\begin{proof}

On the contrary, suppose that $a \in Q(L_p)$ for all primes $p \mid m$. Let $q$ be a prime such that $q\mid 2dL$ but $q\nmid m$. By the assumption that the scale of $L$ is $\Z$, when $q$ is odd there exists an integer $\alpha_q \in Q(L_q)$ such that $q\nmid \alpha_q$; when $q=2$, there exists an odd integer $\alpha_2$ such that $2^{\delta}\alpha_2 \in Q(L_2)$ for either $\delta = 0$ or $\delta =1$. By the Chinese Remainder Theorem, there is a smallest positive integer $\lambda$ such that
    \[
    \begin{cases}
        \lambda \equiv a \,(\text{mod}\,m) \text{ if $m$ is odd, or }\lambda \equiv a \,(\text{mod}\,4m) \text{ if $m$ is even},\\
        \lambda \equiv \alpha_q \,(\text{mod}\,q) \text{ for all odd primes } q \mid dL, q \nmid m,\\
        \lambda \equiv 2^{\delta}\alpha_2 \,(\text{mod}\,2^{\delta +3})  \text{ if } $m$ \text{ is odd}.
    \end{cases}
    \]
For an odd prime $p\mid m$, we then have $\lambda \equiv p^{\ord_pa}a_0 \pmod{p^{\ord_pm}}$ with $p\nmid a_0$ (or $\lambda \equiv 2^{\ord_2a}a_0 \pmod{2^{\ord_2m+2}}$ if $p=2\mid m$). Since the arithmetic progression $\calA_{a,m}$ is assumed to be admissible, $\ord_pa<\ord_pm$ for all primes $p\mid m$ and it thus follows from the Local Square Theorem that $\lambda \in a(\Z_p^{\times})^2$. Therefore $\lambda \in Q(L_p)$, since $a\in Q(L_p)$. The Local Square Theorem also ensures that $\lambda \in Q(L_q)$ for any $q\mid 2dL$ for which $q\nmid m$. Hence we see that $\lambda \in Q(L_p)$ for all primes $p$; that is, $\lambda \in Q(\gen L)$.

Now consider the arithmetic progression $\calA_{\lambda,M}$, where
     \[
    M = 4m \prod_{\begin{tiny} \begin{array}{c} q \mid dL \\ q \nmid m \end{array} \end{tiny}} q.
    \]
Any element $\mu$ of $\calA_{a,m}$ is congruent to $\lambda$ modulo a sufficiently high power of each prime $p\mid dL$ to guarantee that $\mu \in Q(L_p)$; hence, $\calA_{\lambda,M} \subseteq Q(\gen L)$. So by Lemma 1 there exists $\hat{\mu} \in \calA_{\lambda,M}$ such that $\hat{\mu} \in Q(L)$. But this is impossible since $\calA_{\lambda,M} \subseteq \calA_{a,m} \subseteq \calE(L)$.
\end{proof}
We are now in a position to state and prove the main result of this section.

\begin{Prop}
Let $L$ be a positive definite lattice with rank $L \geq 4$. Then
$\calE(L)$ is a finite union of admissible arithmetic progressions if and only if $L$ is regular.
\end{Prop}

\begin{proof}
It remains to prove only the forward implication. So assume $\calE(L)$ is a finite union $\calU$ of admissible arithmetic progressions.
If $2 \mid m$ for some $\calA_{a,m}\subseteq \calU$, note that $\calA_{a,m} = \calA_{a,2m} \cup \calA_{a+m,2m}$.
Since $\calA_{a,m}$ is admissible, we have $\ord_2 (a+m) = \ord_2 a$.
By repeating this process as necessary, we may assume that
    \[
    \ord_2 m - \ord_2 a \geq 3
    \]
for all arithmetic progressions $\calA_{a,m} \subseteq \calU$.
Now suppose that $\lambda \in \calE(L)$. So there exists an arithmetic progression $\calA_{a,m}\subseteq \calU$ such that $\lambda \in \calA_{a,m}$.
Since $\calA_{a,m} \subseteq \calE(L)$, there exists a prime $p \mid m$ such that $a \in \calE(L_p)$ by Lemma 2.
For this prime $p$, there exist integers $a_0, m_0, x$ with $p \nmid a_0m_0$ such that
    \[
    \lambda
    = a + mx
    = p^{\ord_p a} \left( a_0 + p^{\ord_p m - \ord_p a} m_0 x \right).
    \]
Since ${\ord_p m - \ord_p a} \geq 1$ if $p$ is odd and $\geq 3$ if $p=2$, it follows from the Local Square Theorem that there exists $\xi \in \Z_p^{\times}$ such that $\lambda = a\xi^2$. It follows that $\lambda \in \calE(L_p)$.
Hence $\lambda \in \calE(\gen L)$. This establishes that $\calE(L) \subseteq \calE(\gen L)$.
It follows that $\calE(L) = \calE(\gen L)$ and $L$ is regular.
\end{proof}

\section{Diagonal quaternary lattices}

In this section, we return to the special case of diagonal quaternary lattices and prove the results stated in the abstract.

\begin{Prop}
For each $r \geq 1$, $L(r)\cong \qf{1,3,9,9^r}$ avoids exactly $r+1$ arithmetic progressions.
\end{Prop}

\begin{proof}
Let $T \cong \qf{1,3,9}$.
Since $T$ has class number $1$, it was known already to Dickson that
    \[
    \calE(T) = \{ 3\ell + 2 : \ell \geq 0 \} \cup \{ 9^k(9t+6) : k, t \geq 0 \}
    \]
\cite[Chapter V]{leD-39}.
Let $\{ v_1, v_2, v_3, v_4 \}$ be a basis for the lattice $L(r)$ with respect to which the Gram matrix is $\qf{1,3,9,9^r}$, and let
    \[
    K = \Z(3^r v_1) + \Z(3^{r-1} v_3) + \Z v_4 \cong 9^r \qf{1,1,1}.
    \]
So $Q(K_3) \supseteq 9^r \Z_3$, since $\qf{1,1,1}$ is universal over $\Z_3$.
Hence
    \[
    Q(L(r)_3) \supseteq 9^r \Z_3.
    \]
Also $Q(L(r)_p) = \Z_p$ for all primes $p \not=3$.
Using this local information, it can be shown that
    \begin{align*}
    \calE(\gen L(r))
    &= \{ 3\ell + 2 : \ell \geq 0 \} \cup \{ 9^k(9t+6) : t \geq 0, 0 \leq k \leq (r-1) \}\\
    &= \calA_{2,3} \cup \left( \cup_{i=1}^{r} \calA_{2 \cdot 3^{2i-1}, 3^{2i}} \right)\\
    &=\calE(T)\setminus \{9^k(9t+6):t\geq 0, k\geq r\}.
    \end{align*}
So
    \[
    Q(\gen L(r))= Q(T) \cup \{9^k(9t+6):t\geq 0, k\geq r\}.
    \]
Now suppose $\lambda \in Q(\gen L(r))$.
Then either $\lambda \in Q(T)$ or $\lambda = 9^k (9t + 6)$, $k \geq r$, $t \geq 0$.
In the latter case,
    \[
    \lambda - 9^r = 9^r ( 9^{k+1-r} t + 9^{k-r} \cdot 6 - 1 ) \not\in \calE(T).
    \]
So there exists a vector $w \in T$ such that $Q(w) = \lambda - 9^r$.
Hence $\lambda = Q(w + v_4) \in Q( L(r) )$.
We conclude that $\Q(\gen L(r))=Q(L(r))$ and
    \[
    \calE( L(r) )
    = \calA_{2,3} \cup \left( \cup_{i=1}^{r} \calA_{2 \cdot 3^{2i-1}, 3^{2i}} \right).
    \]
Thus, $\calE(L(r))$ can be written as a union of $r+1$ admissible arithmetic progressions.

It remains to show that $\calE(L(r))$ cannot be written as a union of fewer than $r+1$ admissible progressions.
So suppose
    \[
    \calE( L(r) )
    = \calA_{a_1, m_1} \cup \cdots \cup  \calA_{a_t, m_t}
    \]
for some $t \leq r-1$.
Since $2 \in \calE( L(r) )$, we may assume, by re-indexing if necessary, that $2 \in \calA_{a_1, m_1}$.
Since $\calA_{a_1, m_1} \subseteq \calE( L(r) )$, there is a prime $p \mid m_1$ such that $a_1 \in \calE( L(r)_p )$.
As $L(r)_p$ is universal over $\Z_p$ for all $p \not=3$, it follows that $3 \mid m_1$.
So
    \[
    \calA_{a_1,m_1} \subseteq \calA_{a_1,3} = \calA_{2,3}.
    \]
Next note that $6 \in \calE( L(r) )$ and $6 \not\in \calA_{2,3}$; hence $6 \not\in \calA_{a_1,m_1}$, and we may assume, again by re-indexing if necessary, that $6 \in \calA_{a_2,m_2}$.
It follows from $\calA_{a_2,m_2} \subseteq \calE( L(r) )$ that $3 \mid m_2$.
Moreover, $\calA_{a_2, m_2}$ is admissible, so $9 \mid m_2$.
So
    \[
    \calA_{a_2,m_2} \subseteq \calA_{a_2,9} = \calA_{6,9}.
    \]
Continuing in this way, it can be seen that each $\calA_{a_j, m_j}$ must be contained in one of the sets
    \[
    \calA_{2,3}, \hspace{3ex} \calA_{2 \cdot 3^{2i-1}, 3^{2i}}, \hspace{2ex} 1 \leq i \leq (r-1).
    \]
But then the element $2 \cdot 3^{2r-1}$ cannot be contained in any of the sets $\calA_{a_j, m_j}$, $1 \leq j \leq t$, which is a contradiction.
\end{proof}
\noindent\textit{Remark:} In the preceding proof, it is proven that the lattices $L(r)$ are in fact regular. This is also proven in \cite{bmK-pre}, but we have chosen to include the argument here for the sake of completeness.
\smallskip

To conclude the paper, we will determine all of the diagonal quaternary lattices for which the excluded set consists of exactly one arithmetic progression. According to Proposition 2, all such lattices must appear in the list of regular diagonal quaternary lattices determined in \cite{bmK-pre}. This list contains 106 individual lattices and 180 infinite families, arranged so that all those having the same first three coefficients are grouped together. In the following example, we analyze the regular lattices having leading ternary sublattice $\qf{1,1,1}$.

\bigskip

\noindent \textit{Example:} Consider the regular lattices of the type $\qf{1,1,1,d}$. According to \cite{bmK-pre}, the lattices of this type which are regular are:
    \[
    \qf{1,1,1,d}, \text{ with } d = 1,3,5,7 \text{ or } d=x\cdot 2^{2r+1}, \text{ where } r\in \N' \text{ and } x=1,2,3.
    \]
Among this list, the lattices with $d=1,2,3,4,5,6,7$ are known to be universal \cite{mB} and hence have empty excluded set. Also, for the leading ternary sublattice $T=\qf{1,1,1}$, it is a classical result that the excluded set $\calE(T)$ consists of the integers of the type $4^k(8n+7)$ for $k,n\in \N \cup \{0\}$. In particular, $23, 28 \in \calE(T)$. Suppose it were the case that $\calE(L)=\calA_{a,m}$ for some $L\cong \qf{1,1,1,d}$ with $d>28$. Since $23,28 < d$, it follows that $23,28 \in \calE(L)$. Since the difference between $28$ and $23$ is $5$, $m$ would have to be $5$. But this leads to a contradiction since $7\in \calE(L)$ but $12=7+5\in Q(T)\subseteq Q(L)$. So we see that $\calE(L)\neq \calA_{a,m}$ whenever $d=x\cdot 2^{2r+1}$ with $r\geq 2$. So we are left with the three lattices $\qf{1,1,1,d}$ with $d=8,16,24$. In each case, the lattice contains a sublattice that represents all multiples of $4$. For example, $\qf{1,1,1,24}$ contains a sublattice isometric to $4\qf{1,1,1,6}$; since $\qf{1,1,1,6}$ is universal, this yields the result. So in these three cases, the excluded set is precisely $\calA_{7,8}$.

\bigskip

Using similar arguments, an analysis of all the regular diagonal quaternary lattices leads to the following result:

\begin{Prop}
There are exactly 51 diagonal positive definite quaternary lattices $L\cong\qf{a,b,c,d}$ with $1\leq a \leq b \leq c \leq d$ which avoid exactly one arithmetic progression. The lattices and the excluded arithmetic progressions $\calA_{a,m}$ are given in the following list:
{\allowdisplaybreaks
\begin{align*}
\qf{1,1,1,d},  &~~d = 2, 4, 6,                  & \calA_{7,8}\\
\qf{1,1,2,d},  &~~d = 16, 32, 48,               & \calA_{14,16}\\
\qf{1,1,3,d},  &~~d = 9, 18, 27, 36, 45, 54,    & \calA_{6,9}\\
\qf{1,1,4,d},  &~~d = 4, 8, 12, 16, 20, 24, 28, & \calA_{3,4}\\
\qf{1,1,5,8}\,\,  &                                & \calA_{3,8}\\
\qf{1,1,6,d},  &~~d = 9, 18, 17,                & \calA_{3,9}\\
\qf{1,2,2,d},  &~~d = 8, 16, 24,                & \calA_{7,8}\\
\qf{1,2,3,d},  &~~d = 4, 16, 32,                & \calA_{10,16}\\
\qf{1,2,4,d},  &~~d = 16, 32, 48,               & \calA_{14,16}\\
\qf{1,3,3,d},  &~~d = 3, 6, 9, 12, 15, 18,      & \calA_{2,3}\\
\qf{1,3,6,d},  &~~d = 6, 9, 12, 15, 18, 21, 24, 27, 30, & \calA_{2,3}\\
\qf{2,3,3,d},  &~~d = 3, 6, 9,                  & \calA_{1,3}\\
\qf{2,3,6,6}\,\,  &                                & \calA_{1,3}
\end{align*}
}
\end{Prop}


\end{document}